\documentclass[11pt]{article}

\usepackage[margin=1.0in]{geometry}
\usepackage{amsmath,amssymb,amsthm}
\usepackage{booktabs}
\usepackage{array}
\usepackage{longtable}
\usepackage{xcolor}
\usepackage{tikz}
\usetikzlibrary{decorations.markings}
\tikzset{
  fwd/.style={postaction={decorate},decoration={markings,
      mark=at position 0.58 with {\arrow{stealth}}}},
  bwd/.style={postaction={decorate},decoration={markings,
      mark=at position 0.58 with {\arrow{stealth}}}}
}
\usepackage[colorlinks=true,linkcolor=blue!60!black,citecolor=blue!60!black,
            urlcolor=blue!60!black]{hyperref}
\hypersetup{
  pdftitle={The no-flow curves in a family of Lagrangian-Eulerian schemes are
            the space-time integral curves of the Locally Conservative
            Eulerian-Lagrangian Method},
  pdfauthor={Frederico Furtado, Felipe Pereira, Li-Ming Yeh},
  pdfsubject={Comment on arXiv:2106.08363v3 [math.NA]},
  pdfkeywords={Lagrangian-Eulerian schemes; no-flow curves; space-time integral
               curves; integral tubes; local mass conservation; locally
               conservative Eulerian-Lagrangian method; hyperbolic conservation
               laws}}

\newtheorem{proposition}{Proposition}
\theoremstyle{remark}
\newtheorem{remark}{Remark}

\newcommand{\R}{\mathbb{R}}
\newcommand{\dd}{\,\mathrm{d}}

\title{\bfseries The no-flow curves in a family of Lagrangian--Eulerian
schemes are the space--time integral curves of the Locally Conservative
Eulerian--Lagrangian Method\\[7pt]
{\large Comment on arXiv:\hspace{0pt}2106.08363v3 [math.NA]}\\[5pt]
{\normalsize E.~Abreu, A.~Esp\'{\i}rito Santo, W.~Lambert, and J.~P\'erez,
\emph{Convergence of a Lagrangian--Eulerian scheme by a weak asymptotic
analysis for one-dimensional hyperbolic problems}, last revised
2 February 2022;\\[3pt]
published as \emph{Convergence, bounded variation properties and Kruzhkov
solution of a fully discrete Lagrangian--Eulerian scheme via weak asymptotic
analysis for 1D hyperbolic problems},\\
Numer.\ Methods Partial Differential Equations \textbf{39} (2023) 2400--2443,
doi:10.1002/num.22972}}

\author{%
Frederico Furtado\thanks{Department of Mathematics \& Statistics,
University of Wyoming, Ross Hall, 1000 E.\ University Ave, Laramie,
WY 82071, U.S.A.}
\and
Felipe Pereira\thanks{Department of Mathematical Sciences,
The University of Texas at Dallas, Richardson, TX 75080, U.S.A.
(\texttt{luisfelipe.pereira@utdallas.edu}). Corresponding author.}
\and
Li-Ming Yeh\thanks{Department of Applied Mathematics, National Yang Ming
Chiao Tung University, 1001 Ta Hsueh Road, Hsinchu 30010, Taiwan, ROC.}}

\date{}

\begin{document}
\maketitle

\begin{abstract}
This Comment concerns the work posted as arXiv:\hspace{0pt}2106.08363v3
[math.NA] and published, under a different title, as Abreu, Esp\'{\i}rito
Santo, Lambert, and P\'erez, \emph{Numer.\ Methods Partial Differential
Equations} \textbf{39} (2023) 2400--2443; all quotations and all page,
section, and equation references below are to the published version.  That
article builds its scheme on space--time control volumes whose lateral
boundaries it calls \emph{no-flow curves}, solving $\dot\sigma=H(u)/u$, and
presents this concept as new: its abstract invokes ``the improved concept of
no-flow curves, as introduced by the authors.''  We show that the continuous
object so named is the \emph{space--time integral curve} of the locally
conservative Eulerian--Lagrangian method of Douglas, Pereira, and Yeh
[Comput.\ Geosci.\ \textbf{4} (2000) 1--40], and that the no-flow region is
the space--time tube it bounds: after the scalar dictionary is applied, the two
curve equations, the zero-flux property of the lateral boundaries, and the
exact local balance coincide term by term, and the nonlinear scalar
forward-tracked form is already that of Mancuso, Pereira, and de Souza
[TEMA \textbf{8} (2007) 269--276, 277--286].

What is at issue is not that a control volume is transported, which is generic
to moving-mesh methods, but \emph{which} curve transports it.  For a scalar
conservation law no velocity is given by the model, and the two candidate
families --- the characteristic $\dot x=H'(u)$ and the flux-to-state ratio
$\dot x=H(u)/u$ --- are distinct speed functions, coinciding identically on an
interval of states only for a linear flux $H(u)=cu$.  Where $u\neq0$, the
flux-to-state ratio is precisely the condition of zero normal mass flux, so
local conservation on the resulting tube is an identity of the geometry rather
than a property of a discretization.  A short
proposition records that equivalence, under the normalization $H(0)=0$.
Making and arguing that choice, against the characteristic alternative, is the
content of Section~5 of Douglas, Pereira, and Yeh, and it is recognized as
prior art in the same field by unrelated authors.  The article under comment
itself states (p.~2401) that this work was ``the first work in the literature
to introduce a space--time local conservation,'' with the region called the
integral tube and bounded by integral curves.  Our conclusion is that these
two statements concern the same object.  Nor is that article an isolated
case.  In the full-text corpus documented here, the same construction is
principally called the integral tube and credited to Douglas, Pereira, and Yeh
through the HYP2018 proceedings, published in 2020; the phrase ``no flow
curves'' already appears in 2019 for the boundaries of those same tubes; from
2021 on, that label becomes dominant, and the construction is described as
substantially different from, or presented as introduced by, the later
authors; papers published in 2025 and 2026 use the construction while citing
later papers, but not Douglas, Pereira, and Yeh; and a different 2026 paper
gives mixed attribution, calling the no-flow curve an ``extension'' of the DPY
integral curve while its own equations show that the two curves are identical
and that diffusion and dispersion do not modify them.  The continuous
identification made here applies to every paper tabulated.  It leaves untouched
the article's fully discrete non-staggered scheme, its weak CFL condition, and
its analysis of that scheme by the weak asymptotic method of Danilov,
Omel'yanov, and Shelkovich.
\end{abstract}



\section{Purpose and scope}\label{sec:intro}

\paragraph{Work under comment and citation convention.}
This Comment concerns the work posted by E.~Abreu, A.~Esp\'{\i}rito Santo,
W.~Lambert, and J.~P\'erez as arXiv:\hspace{0pt}2106.08363v3 [math.NA], last
revised 2 February 2022, under the title \emph{Convergence of a
Lagrangian--Eulerian scheme by a weak asymptotic analysis for one-dimensional
hyperbolic problems}.  That is the arXiv record to which this Comment is
intended to be linked.  The same work was subsequently published, under the
different title \emph{Convergence, bounded variation properties and Kruzhkov
solution of a fully discrete Lagrangian--Eulerian scheme via weak asymptotic
analysis for 1D hyperbolic problems}, as \cite{AESLP2023}.  Because the
published article is the version of record, \emph{this Comment refers
throughout to the published paper}: the citation \cite{AESLP2023}, every
quotation, and every page, section, and equation number below are taken from
\emph{Numer.\ Methods Partial Differential Equations} \textbf{39} (2023)
2400--2443, doi:10.1002/num.22972, and not from the preprint.  The arXiv
identifier and preprint title serve only to identify the work and to locate
this Comment on the arXiv record.

\paragraph{The construction at issue, and two statements about it.}
The article under comment constructs and analyzes a fully discrete
Lagrangian--Eulerian scheme for the scalar problem
\begin{equation}\label{eq:ivp}
\frac{\partial u}{\partial t}+\frac{\partial H(u)}{\partial x}=0,
\qquad x\in\R,\ t>0,\qquad u(x,0)=u_0(x).
\end{equation}
Its space--time control volumes $D^{n,n+1}_j$ have lateral boundaries
$\sigma^n_{j\pm1/2}$ that solve $\dot\sigma=H(u)/u$, are described as
``naturally impervious zero-flux boundaries,'' and are called \emph{no-flow
curves}; the mass inside each control volume is therefore exactly conserved.
The article presents this object as new.  Its abstract states that the
approach ``is based on the improved concept of no-flow curves, as introduced
by the authors''; pages~2401 and 2405 describe a formulation ``based on the new and
substantial improvement interpretation of the integral tube'' now ``called
no-flow curves''; and page~2402 refers to ``the new concept no-flow curves''
as ``substantially different in theory foundations'' from the Locally
Conservative Eulerian--Lagrangian Method (LCELM) of Douglas, Pereira, and Yeh
\cite{DPY2000}.

The same article also states, on page~2401, that \cite{DPY2000} was ``the
first work in the literature to introduce a space--time local conservation,''
with the corresponding region called the \emph{integral tube} and bounded by
\emph{integral curves}, and records that its own Section~2 uses ``the same
mass conservation property'' as \cite{DPY2000}.  The purpose of this Comment
is to establish that these two groups of statements are about one and the same
continuous object, and that the second is consistent with the documented
lineage.

\paragraph{A changing attribution.}
The article under comment states the claim most explicitly, but it does not
state it alone, and the record is worth setting out at the start rather than
leaving to a chronology at the end.  Table~\ref{tab:attrib} is a defined
documentary corpus: it includes papers by overlapping author teams whose full
texts we examined and in which the same curve equation, zero-flux boundary, or
integral-tube construction appears.  For each, the table reproduces the
relevant attribution language.\footnote{The attribution of prior ideas is
not a discretionary courtesy but a norm of scholarly publishing. Wiley, the
publisher of the journal in which the paper under comment appeared, states
that the use of another person's ideas, data, or text requires proper
attribution, and identifies among violations of publication ethics the
failure to give due credit to prior work and the neglect of related work in a
manner that increases the apparent novelty of the results
\cite{WileyEthics}. The European Code of Conduct for Research Integrity
likewise requires authors to acknowledge important work and contributions of
others, identifies selective or inaccurate citation as an unacceptable
practice, and requires proper credit to the original source when the work or
ideas of others are used \cite{ALLEA2023}. The present Comment makes no
allegation of misconduct: it establishes a mathematical identification in
Section~\ref{sec:three} and Proposition~\ref{prop:flux}, and records in
Table~\ref{tab:attrib} the citation practice that, in the authors' view,
follows from it.}  It is not presented as an exhaustive
bibliography of every application.  Three phases are visible in the verified
corpus.

Through the HYP2018 proceedings, published in 2020, the region is principally
called the \emph{integral tube}, and the construction is credited to
\cite{DPY2000}.  The 2017 journal paper \cite{ALPS2017} and the 2017 CNMAC
paper \cite{APS-CNMAC2017} both identify \cite{DPY2000} as the work in which
the relevant space--time mass-conservation region was established for a
scalar convection-dominated nonlinear parabolic problem.  The 2018 CNMAC paper
\cite{APS-CNMAC2018} and the HYP2018 proceedings paper
\cite{ALPS-HYP2020} again define the same parameterized integral curves and
zero-flux volumes while citing \cite{DPY2000}; Figure~1 of the latter is
captioned ``The Integral tube.''  The 2019 paper \cite{AbreuPerez2019} states
plainly that the work uses ``the concept of integral tubes (or space--time
tubes) introduced in'' \cite{DPY2000}.  Its terminology is already
transitional: the caption of Figure~5 describes ``the first-order
approximation of the no flow curves of the integral tubes.''  Thus the phrase
\emph{no flow curves} does not first appear in 2021.  From 2021 it becomes
the dominant label, and the attribution changes:
\cite[p.~6]{AMPR2021} records that ``in the terminology discussed in
[24,25], the region $D^n_j$ is called integral tube, which is substantially
different from our recent developments''; reference [25] there is
\cite{DPY2000}, while [24] is \cite{DouglasHuang2001}.  The article under
comment describes the concept as ``introduced by the authors.''  In 2025 the
construction is used under the new name in \cite{AdlCJL2025} with no reference
to \cite{DPY2000} at all; a reader wanting ``a detailed description of no-flow
curves'' is referred there to \cite{AdlCLP2022}.  The 2026 application paper
\cite{AFFG2026} likewise contains no reference to DPY: it calls its
semi-discrete formulation ``recently introduced,'' says that the method is
based on ``no-flow'' curves while citing \cite{AbreuPerez2019}, and uses the
same componentwise flux-to-state ratios.  A different 2026 paper,
\cite{ALL2026dd}, restores an explicit reference to DPY, but describes the
no-flow curve as a natural \emph{extension} of the integral curve ``originally
introduced by Douglas, Pereira, and Yeh.''  Its own Definition~2.1 and
equations (9) and (14) instead give exactly the same zero-flux curve: the
space--time field is $(H(k(x),u),u)^{\mathrm T}$ and
$\dot\sigma=H(k(\sigma),u)/u$.  The paper then states that its diffusive and
dispersive terms do not modify the vector field that defines those curves.
What is extended is therefore the equation, scheme, and analysis, not the
continuous curve.

Two features of the sequence of articles in Table 2 should be stated precisely, because they are
easily misdescribed.  First, \cite{DPY2000} is not generally omitted: it is
cited in every paper tabulated except \cite{AdlCJL2025,AFFG2026}; in
\cite{AbreuPerez2019} and \cite{APL2025bal} it is named as a source of the
construction, and \cite{ALL2026dd} names DPY as the source of the integral
curve but casts the identical no-flow curve as an extension.  The second feature is what DPY
is cited for matters.  In \cite[p.~3]{AFLP-JCAM2022} the lateral boundaries are
``two special curves [27] (here under the name no-flow curves [6])'';
references [27] and [6] there are \cite{DPY2000} and
\cite{AbreuPerez2019}, respectively.  On the same page, however, the paper
presents a scheme based on ``the novel concept of no-flow curves recently
introduced in the literature [5,6,28],'' all three references being to papers
by overlapping author teams.  The question raised by the sequence is
therefore not only one of missing references.  More commonly, a construction
that is cited is at the same time presented as new.

\paragraph{Plan.}
Section~\ref{sec:three} places the defining equations of three constructions
side by side --- \cite{DPY2000} (2000), \cite{MPS2007,MPS2007b} (2007), and
\cite{AESLP2023} (2023) --- and shows that, in the scalar source-free setting,
they are the same construction in different notation.
Section~\ref{sec:which} isolates what is specific about that construction and
therefore what it means to introduce it: not that a control volume moves, but
which curve moves it.  Section~\ref{sec:prop} records, in one proposition with
a two-line proof, why the flux-to-state ratio is exactly the zero-flux
condition; this is a standard computation, offered as a clarification and not
as a new theorem, and it carries a normalization caveat that we state
explicitly.  Section~\ref{sec:chronology} follows the terminology, and
Section~\ref{sec:new} sets out what is genuinely new in \cite{AESLP2023}.

Three limitations of scope are stated at the outset and are not relaxed
anywhere below.
\begin{itemize}
\item[(i)] Our claim is confined to the Eulerian--Lagrangian literature for
transport in porous media and for hyperbolic conservation laws --- the arena
in which \cite{AESLP2023} situates itself and from which it draws the
antecedents it cites.  Section~\ref{sec:which} explains why the generic
moving-mesh, or arbitrary Lagrangian--Eulerian, template is a different kind
of statement and is not what is at issue; beyond the specific points made
there, we do not survey mass-coordinate, particle, front-tracking, or
stream-tube methods.
\item[(ii)] The correspondence we assert is between \emph{defining equations
of continuous constructions}.  We prove no existence, uniqueness, or selection
result for the curves, and we make no claim about the discrete curves produced
by the respective algorithms, which select values of $u$ on the curves by
different reconstructions.
\item[(iii)] Nothing here questions the correctness of the numerical method
or of the convergence analysis of \cite{AESLP2023}.  We do note, in
Remark~\ref{rem:vacuum}, where the hypotheses of that analysis exclude the
vacuum state; that is an observation about its scope, not about its
correctness.  What in that article is new, and what it takes from others with
due citation, is set out in Section~\ref{sec:new}.
\end{itemize}

\section{The three constructions, equation by equation}\label{sec:three}

\paragraph{2000: Douglas, Pereira, and Yeh \cite{DPY2000}.}
Section~5 of \cite{DPY2000} treats the transport subproblem of a waterflood
system, written \emph{in divergence form} in space--time,
\begin{equation}\label{eq:dpy-div}
\nabla_{t,x}\cdot\bigl(\Phi S,\ \Lambda_{\mathrm w}U\bigr)^{\mathrm T}
= q^{+}-\Lambda_{\mathrm w}q^{-}
\qquad\text{(eq.~(5.2) of \cite{DPY2000})},
\end{equation}
with $\Phi S$ the conserved quantity and $\Lambda_{\mathrm w}U$ its flux.
The Introduction of \cite{DPY2000} identifies the divergence form, in contrast
to the nondivergence form underlying MMOC \cite{DouglasRussell1982} and
MMOCAA \cite{DFP1997,DHP1999} --- the latter developed in work co-authored by
two of the three authors of \cite{DPY2000} --- as the step that ``allows the
localization of the transport.''  For each boundary point of a
set $\mathcal K$ the construction solves the final value problem
\begin{equation}\label{eq:dpy-curve}
\frac{dy}{dt}=\frac{\Lambda_{\mathrm w}U}{\Phi S},
\qquad y(x;t_{n,\kappa+1})=x
\qquad\text{(eq.~(5.4a) of \cite{DPY2000})} :
\end{equation}
the slope is the ratio of the flux to the conserved quantity.  These curves
bound the space--time tube $\mathcal D$ of figure~1 of \cite{DPY2000}, which
calls the curves \emph{integral curves} and the region simply ``the tube''
(once, in the general discussion of \S5.1, ``the space--time tubes''); the
name \emph{integral tube} is the one used in
\cite{MPS2007} (``tubo integral'') and is the name \cite{AESLP2023} itself
attributes to \cite{DPY2000} (p.~2401); the zero-flux property is stated there explicitly, the
outward normal on the lateral surface being ``orthogonal to the vector
$(\Phi S,\Lambda_{\mathrm w}U)^{\mathrm T}$''; and integrating
\eqref{eq:dpy-div} over $\mathcal D$ gives the exact local balance
(eqs.~(5.6)--(5.7) of \cite{DPY2000}).  Two further points of that section
carry the weight of Section~\ref{sec:which} below.  The integral curves are
contrasted with the \emph{characteristic} curves
$dy/dt=\Lambda_{\mathrm w}'U/\Phi$ of MMOC, MMOCAA, ELLAM
\cite{CRHE1990}, and characteristics-mixed methods
\cite{ArbogastWheeler1995} (eqs.~(5.9a)--(5.9b), and (5.11a) for the
corresponding tube), the two families of curves coinciding identically only for linear
$\Lambda_{\mathrm w}$, the tube $\mathcal D$ being
``associated with the transport of mass'' whereas the characteristic tube
transports saturation; and equation (5.13) identifies the interface
distribution term that a locally conservative scheme built on characteristics
generates and whose accurate evaluation ``poses a severe problem.''  The
behavior at vanishing saturation is also addressed: since
$\Lambda_{\mathrm w}(S)/S=0$ for $0<S\le S_{\mathrm{w,res}}$ (eq.~(5.8)),
``the integral curves do not degenerate as $S\to 0$.''

\paragraph{2007: the scalar, forward-tracked specialization
\cite{MPS2007,MPS2007b}.}
Mancuso, Pereira, and de Souza specialized this construction to the scalar
conservation law \eqref{eq:ivp}, in forward-in-time form.  With $f\equiv H$,
the law is written as $\nabla_{t,x}\cdot(u,f(u))^{\mathrm T}=0$ (eq.~(2.2) of
\cite{MPS2007}) and the lateral boundaries solve the \emph{initial} value
problem
\begin{equation}\label{eq:mps-curve}
\frac{dy}{dt}=\frac{f(u)}{u},\qquad y(x,t^{n})=x
\qquad\text{(eq.~(2.12) of \cite{MPS2007})}.
\end{equation}
The bounded region is again called the integral tube (``tubo integral'') and
its lateral boundaries the integral curves (``curvas integrais''); their
normal is orthogonal to $(u,f(u))^{\mathrm T}$, and integration over the tube
gives exact local conservation (eqs.~(2.14)--(2.15)).  The resulting fully
discrete method, FLCELM (\emph{Forward Locally Conservative
Eulerian--Lagrangian Method}), tracks curves forward from cell vertices
(eq.~(2.16)), evolves cell averages (eqs.~(2.17)--(2.19)), carries a CFL-type
restriction involving both $f'$ and the ratio $f(U)/U$ (eq.~(2.21)), and is
tested on the Burgers and Buckley--Leverett \cite{BuckleyLeverett1942}
equations; the variant FLCELM-R adds a MinMod piecewise-linear
reconstruction of MUSCL type \cite{vanLeer1979}.  Both papers credit the source
of the geometry: the tubes are built ``usando a estrat\'egia de
constru\c{c}\~ao dos tubos no espa\c{c}o-tempo do m\'etodo LCELM'' [using the
strategy of construction of the space--time tubes of the LCELM method], with
\cite{DPY2000} cited, and \cite{MPS2007b} records that the construction
``uses the local conservation identity that appears in'' \cite{DPY2000}.
Forward tracking of integral tubes, under that name, appears in the same
period in the linear-transport scheme of Aquino, Pereira, Amaral Souto, and
Francisco \cite{QPAS2007}, which introduced Forward Integral-Tube Tracking
(FIT); in that linear setting integral curves and characteristics coincide, so
\cite{QPAS2007} is not an antecedent of the nonlinear scheme of
\cite{AESLP2023}, but it fixes the date at which forward tracking of integral
tubes, and the phrase itself, were in use.

\paragraph{2023: the article under comment \cite{AESLP2023}.}
There the scalar law is written as $\nabla_{x,t}\cdot(H(u),u)^{\mathrm T}=0$
(eq.~(2)), the control volumes $D^{n,n+1}_j$ are defined in eq.~(3), and
their lateral boundaries solve
\begin{equation}\label{eq:noflow}
\frac{d}{dt}\sigma^n_{j\pm1/2}(t)
=\frac{H\bigl(u(\sigma^n_{j\pm1/2}(t),t)\bigr)}
      {u\bigl(\sigma^n_{j\pm1/2}(t),t\bigr)},
\qquad t^n<t\le t^{n+1}
\qquad\text{(eq.~(6) of \cite{AESLP2023})},
\end{equation}
yielding the local balance (eq.~(4)).  Under the scalar dictionary
\begin{equation}\label{eq:dict}
\Phi S\longleftrightarrow u,\qquad
\Lambda_{\mathrm w}U\longleftrightarrow H(u),
\end{equation}
the source-free scalar form of \eqref{eq:dpy-div} is eq.~(2) of
\cite{AESLP2023} and eq.~(2.2) of \cite{MPS2007}, and \eqref{eq:dpy-curve}
becomes
\begin{equation}\label{eq:scalar-curve}
\frac{dy}{dt}=\frac{H(u)}{u},
\end{equation}
which is \eqref{eq:mps-curve} and \eqref{eq:noflow}.
Table~\ref{tab:correspondence} lists the correspondence element by element.

\paragraph{The three constructions drawn side by side.}
Figure~\ref{fig:tubes} reproduces, in a common notation, the schematic
space--time picture that accompanies each of the three constructions:
figure~1 of \cite{DPY2000}, figure~2 of \cite{MPS2007} (``Tubos integrais no
espa\c{c}o-tempo''), and figure~1 of \cite{AESLP2023}.  We include it because
the equation-level comparison of Table~\ref{tab:correspondence} is easy to
read as a numerical coincidence between formulas taken out of context,
whereas the three figures show that the formulas carry the same geometric
role in each paper: in every panel the shaded set is a space--time control
volume, its two lateral boundaries are curves whose slope is the ratio of the
flux to the conserved quantity, no mass crosses those boundaries, and the
mass on the bottom edge therefore equals the mass on the top edge.

What the panels show as genuinely different is not the geometry but the
\emph{direction in which the boundaries are tracked}: \cite{DPY2000} solves
the final value problem \eqref{eq:dpy-curve} and traces the curves backwards
from the later time level, whereas \cite{MPS2007,MPS2007b} and
\cite{AESLP2023} solve the initial value problems \eqref{eq:mps-curve} and
\eqref{eq:noflow} and trace them forwards --- the change of direction being
already the 2007 step, as \cite{QPAS2007} shows for linear advection.  The
figure is a picture of the continuous geometry only.  It says nothing about
how any of the three papers discretizes that geometry, which is where they
differ substantively (Sections~\ref{sec:prop} and \ref{sec:new}); in
particular the reader should not read the smooth curves of the panels as the
polygonal curves that the algorithms actually compute, for the reasons given
in Remark~\ref{rem:limits}.

\begin{figure}[htb]
\centering
\begin{tikzpicture}[scale=0.74,
   every node/.style={scale=0.80},
   pcap/.style={align=center,text width=4.7cm,font=\footnotesize}]
\begin{scope}
\fill[blue!8] (0,0) .. controls (0.55,1.0) .. (0.95,2.2)
 -- (3.45,2.2) .. controls (3.05,1.0) .. (2.5,0) -- cycle;
\draw[->] (-0.9,0) -- (4.3,0) node[below] {$x$};
\draw[->] (-0.7,-0.15) -- (-0.7,2.9) node[left] {$t$};
\draw[thick,bwd] (0.95,2.2) .. controls (0.55,1.0) .. (0,0);
\draw[thick,bwd] (3.45,2.2) .. controls (3.05,1.0) .. (2.5,0);
\draw[very thick] (0,0) -- (2.5,0);
\node[below=1pt] at (1.25,0) {$\overline{\mathcal K}$};
\draw[very thick] (0.95,2.2) -- (3.45,2.2);
\node[above=1pt] at (2.2,2.2) {$\mathcal K$};
\node at (1.9,1.05) {$\mathcal D$};
\node[left] at (-0.75,0) {$t_{n,\kappa}$};
\node[left] at (-0.75,2.2) {$t_{n,\kappa+1}$};
\draw[dotted] (-0.7,2.2) -- (0.95,2.2);
\node at (1.7,-1.35)
 {$\dfrac{dy}{dt}=\dfrac{\Lambda_{\mathrm w}U}{\Phi S}$ \ (5.4a)};
\node[pcap] at (1.7,-2.65)
 {\cite{DPY2000} (2000), fig.~1\\[1pt]
  \emph{integral curves}; the tube $\mathcal D$\\[1pt]
  final-value tracking};
\end{scope}
\begin{scope}[xshift=6.5cm]
\fill[blue!8] (0,0) .. controls (0.55,1.0) .. (0.95,2.2)
 -- (3.45,2.2) .. controls (3.05,1.0) .. (2.5,0) -- cycle;
\draw[->] (-0.9,0) -- (4.3,0) node[below] {$x$};
\draw[->] (-0.7,-0.15) -- (-0.7,2.9) node[left] {$t$};
\draw[thick,fwd] (0,0) .. controls (0.55,1.0) .. (0.95,2.2);
\draw[thick,fwd] (2.5,0) .. controls (3.05,1.0) .. (3.45,2.2);
\draw[very thick] (0,0) -- (2.5,0);
\node[below=1pt,font=\footnotesize] at (1.25,0)
 {$[x_{j-1/2},x_{j+1/2}]$};
\draw[very thick] (0.95,2.2) -- (3.45,2.2);
\node[align=center,font=\footnotesize] at (1.9,1.05) {tubo\\integral};
\node[left] at (-0.75,0) {$t^{n}$};
\node[left] at (-0.75,2.2) {$t^{n+1}$};
\draw[dotted] (-0.7,2.2) -- (0.95,2.2);
\node at (1.7,-1.35)
 {$\dfrac{dy}{dt}=\dfrac{f(u)}{u}$ \ (2.12)};
\node[pcap] at (1.7,-2.65)
 {\cite{MPS2007,MPS2007b} (2007), fig.~2\\[1pt]
  \emph{integral curves, integral tube}\\[1pt]
  forward tracking};
\end{scope}
\begin{scope}[xshift=13.0cm]
\fill[blue!8] (0,0) .. controls (0.55,1.0) .. (0.95,2.2)
 -- (3.45,2.2) .. controls (3.05,1.0) .. (2.5,0) -- cycle;
\draw[->] (-0.9,0) -- (4.3,0) node[below] {$x$};
\draw[->] (-0.7,-0.15) -- (-0.7,2.9) node[left] {$t$};
\draw[thick,fwd] (0,0) .. controls (0.55,1.0) .. (0.95,2.2)
 node[pos=0.70,left=1pt,font=\footnotesize] {$\sigma^n_{j-1/2}$};
\draw[thick,fwd] (2.5,0) .. controls (3.05,1.0) .. (3.45,2.2)
 node[pos=0.70,right=1pt,font=\footnotesize] {$\sigma^n_{j+1/2}$};
\draw[very thick] (0,0) -- (2.5,0);
\node[below=1pt,font=\footnotesize] at (1.25,0)
 {$[x^{n}_{j-1/2},x^{n}_{j+1/2}]$};
\draw[very thick] (0.95,2.2) -- (3.45,2.2);
\node[font=\footnotesize] at (1.9,1.05) {$D^{n,n+1}_j$};
\node[left] at (-0.75,0) {$t^{n}$};
\node[left] at (-0.75,2.2) {$t^{n+1}$};
\draw[dotted] (-0.7,2.2) -- (0.95,2.2);
\node at (1.7,-1.35)
 {$\dfrac{d\sigma}{dt}=\dfrac{H(u)}{u}$ \ (6)};
\node[pcap] at (1.7,-2.65)
 {\cite{AESLP2023} (2023), fig.~1\\[1pt]
  \emph{no-flow curves, no-flow region}\\[1pt]
  forward tracking};
\end{scope}
\end{tikzpicture}
\caption{The same continuous construction as it is drawn in the three
sources, under the dictionary \eqref{eq:dict}.  In each panel the lateral
boundaries of the shaded space--time region have slope equal to the ratio of
the flux to the conserved quantity, carry no mass flux
(Proposition~\ref{prop:flux} of Section~\ref{sec:prop}, under the flux
normalization $H(0)=0$ adopted there), and the mass on the lower edge equals
the mass on the upper edge (Remark~\ref{rem:tube}).  The arrows indicate the
direction of tracking, the only structural difference among the three panels:
backwards from $t_{n,\kappa+1}$ in \cite{DPY2000}, forwards from $t^n$ in
\cite{MPS2007,MPS2007b} and \cite{AESLP2023}.  Names differ across panels;
the equations do not.}
\label{fig:tubes}
\end{figure}
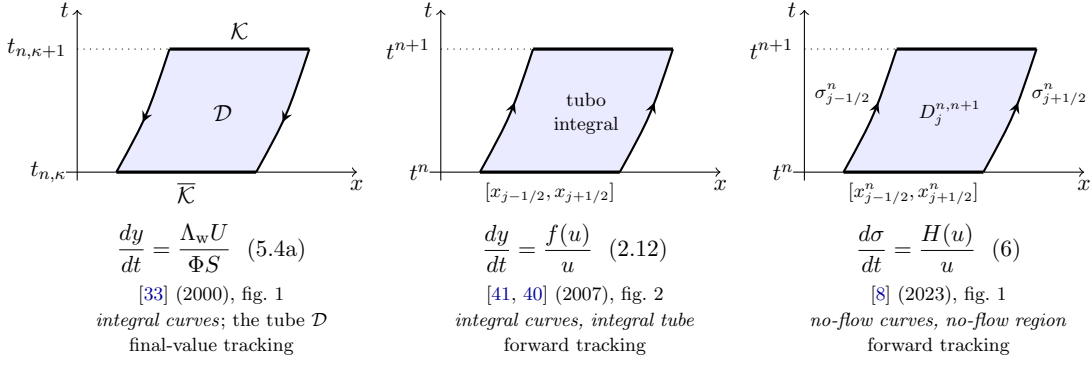

\begin{table}[htb]
\centering
\caption{Element-by-element comparison of the three continuous
constructions, under the dictionary \eqref{eq:dict}.  The divergence row is
the source-free scalar specialization of \cite{DPY2000}; note also that
\cite{DPY2000,MPS2007} order the space--time field as $(t,x)$ and
\cite{AESLP2023} as $(x,t)$.  The comparison uses the same flux function
throughout and is therefore unaffected by the flux normalization adopted in
Section~\ref{sec:prop}.}
\label{tab:correspondence}
\small
\renewcommand{\arraystretch}{1.25}
\begin{tabular}{%
>{\raggedright\arraybackslash}p{0.20\textwidth}
>{\raggedright\arraybackslash}p{0.235\textwidth}
>{\raggedright\arraybackslash}p{0.23\textwidth}
>{\raggedright\arraybackslash}p{0.23\textwidth}}
\toprule
\textbf{Element} & \textbf{\cite{DPY2000} (2000)} &
\textbf{\cite{MPS2007,MPS2007b} (2007)} &
\textbf{\cite{AESLP2023} (2023)}\\
\midrule
Space--time divergence form &
$\nabla_{t,x}\!\cdot(\Phi S,\Lambda_{\mathrm w}U)^{\mathrm T}\!=0$ (5.2) &
$\nabla_{t,x}\!\cdot(u,f(u))^{\mathrm T}\!=0$ (2.2) &
$\nabla_{x,t}\!\cdot(H(u),u)^{\mathrm T}\!=0$ (2)\\
Curve equation (slope $=$ flux\,/\,conserved quantity) &
$dy/dt=\Lambda_{\mathrm w}U/(\Phi S)$ (5.4a), final value &
$dy/dt=f(u)/u$ (2.12), initial value &
$d\sigma/dt=H(u)/u$ (6), initial value\\
Name of the curve &
integral curve &
curva integral / integral curve &
no-flow curve\\
Zero lateral flux &
normal $\perp(\Phi S,\Lambda_{\mathrm w}U)^{\mathrm T}$ on
$\partial\mathcal D$ &
normal $\perp(u,f(u))^{\mathrm T}$ &
``naturally impervious zero-flux boundaries''\\
Space--time control volume &
\emph{integral curves}; the tube $\mathcal D$ (fig.~1) &
tubo integral (fig.~2) &
no-flow region $D^{n,n+1}_j$ (3)\\
Exact local balance &
(5.6)--(5.7) &
(2.14)--(2.15) &
(4)\\
Explicit contrast with characteristics &
(5.9a)--(5.9b), (5.11a); interface distribution term (5.13) &
time-step bound involves both $f'$ and $f(U)/U$ (2.21) &
p.~2402, improvement over ``backward tracking \dots\ of the characteristic
curves''\\
Behavior at a vanishing conserved quantity &
$\Lambda_{\mathrm w}(S)/S=0$ on $0<S\le S_{\mathrm{w,res}}$ (5.8); no
degeneracy as $S\to0$ &
not separately discussed &
$f^\pm$ splitting after setting $f=H(u)/u$; no general prescription at
$u=0$; analysis assumes $u>a>0$\\
\bottomrule
\end{tabular}
\end{table}

\section{Which curve: what it means to introduce this
construction}\label{sec:which}

A comparison of equations invites the objection that moving control volumes
are old and common property.  They are, and the objection has to be met
precisely, because meeting it is what identifies the contribution actually at
stake.

\paragraph{Transporting a control volume is generic; choosing the curve is
not.}
That a mesh may be moved with the flow, and the result projected back onto a
fixed grid, is the defining freedom of arbitrary Lagrangian--Eulerian methods
\cite{HirtAmsdenCook1974} and is standard in computational gas dynamics
\cite{ColellaWoodward1984}.  Nothing below is a claim about that template.
But in the systems for which it was devised, the transporting velocity is
\emph{supplied by the model}: the Euler equations carry a velocity $v$ among
their unknowns, the continuity equation reads $\rho_t+(\rho v)_x=0$, and
moving with the material means moving with $v$.  The ratio of the mass flux
$\rho v$ to the density $\rho$ is the velocity already present in the
formulation, so no choice is made and none needs to be justified.

For the scalar conservation law \eqref{eq:ivp} there is no such variable.  A
curve along which to transport a cell must be \emph{defined} from the
solution itself, and the two natural definitions do not agree:
\begin{equation}\label{eq:two-families}
\frac{dx}{dt}=H'(u)
\quad\text{(characteristic)}
\qquad\text{versus}\qquad
\frac{dx}{dt}=\frac{H(u)}{u}
\quad\text{(flux-to-state ratio)},
\end{equation}
These are distinct functions of the state: by
$\frac{d}{du}\bigl(H(u)/u\bigr)=\bigl(uH'(u)-H(u)\bigr)/u^{2}$, they agree
throughout an interval of states precisely when $H(u)/u$ is constant there,
that is, for a linear flux $H(u)=cu$ --- although for a nonlinear flux they
may of course agree at isolated states.  Where $u\neq0$, the flux-to-state
ratio is precisely the condition of zero normal mass flux
(Proposition~\ref{prop:flux}).  The first is the classical choice
and the one taken by the modified method of characteristics, by MMOCAA, by
ELLAM, and by the characteristics-mixed methods
\cite{DouglasRussell1982,DFP1997,DHP1999,CRHE1990,ArbogastWheeler1995}.  The second is the choice
for which the lateral mass flux vanishes identically, so that exact local
conservation is an \emph{identity of the geometry} rather than a property to
be arranged in a discretization; Proposition~\ref{prop:flux} and
Remark~\ref{rem:tube} state that equivalence and its consequence.  Selecting
the second family, in a setting where the two differ and where the classical
literature had settled on the first, is the substantive act.

\paragraph{Where that act is performed.}
It is performed in Section~5 of \cite{DPY2000}, and performed with its
reasons.  The ratio curves are introduced (5.4a); the tube they bound is
exhibited (figure~1); the orthogonality of the outward normal to
$(\Phi S,\Lambda_{\mathrm w}U)^{\mathrm T}$ is stated; the exact local balance
is derived (5.6)--(5.7); the curves are set against the characteristic family
(5.9a)--(5.9b), (5.11a), with the tube $\mathcal D$ identified as the one ``associated with
the transport of mass'' while the characteristic tube transports saturation;
and the price of the classical alternative is identified, in the interface
distribution term (5.13) that a locally conservative scheme built on
characteristics generates and whose accurate evaluation ``poses a severe
problem.''  That deliberate contrast is what distinguishes an introduction of
the construction from an incidental use of a moving mesh.  Where the velocity
is given, the question which curve does not arise; where it is not given,
answering it is the whole of the construction.

\paragraph{Recognition of the construction in the same field.}
The point is not ours alone.  Working independently of the lineage discussed
here, Arbogast and Huang \cite{ArbogastHuang2010} and Arbogast, Huang, and
Hung \cite{ArbogastHuangHung2012} develop exactly conservative
Eulerian--Lagrangian and stream-tube methods in which a space--time control
volume is traced along the flow and carries no normal flux across its lateral
boundaries.  Their setting is advection--diffusion with a prescribed velocity
field rather than a nonlinear scalar law, so they are not antecedents of
\cite{AESLP2023} and we do not present them as such.  We cite them for one
documentary purpose: \cite{ArbogastHuangHung2012} cites \cite{DPY2000}.  An
unrelated group building locally conservative Eulerian--Lagrangian methods in
this arena treats the LCELM construction as part of the established prior art
of the subject.

\paragraph{The claim, stated exactly.}
Within the Eulerian--Lagrangian literature for transport in porous media and
for hyperbolic conservation laws, the space--time curve of slope flux over
conserved quantity, the tube it bounds, the vanishing of the lateral mass
flux, and the exact local balance that follows enter with \cite{DPY2000},
which introduces them, argues for them against the characteristic
alternative, and names them; and their nonlinear scalar forward-tracked form
enters with \cite{MPS2007,MPS2007b}, under the same names and with
\cite{DPY2000} credited.  This is the claim of the present Comment, and it is
the claim that Section~\ref{sec:three} verifies equation by equation.

\section{Why the ratio equation is the zero-flux condition}\label{sec:prop}

For a curve $x=\gamma(t)$, we write $u_\gamma(t):=u(\gamma(t),t)$ at
points of continuity and use the same notation for the relevant one-sided
trace at discontinuities.  The following proposition is elementary and is
included so that the phrase ``no-flow'' and the equation
$\dot\gamma(t)=H(u_\gamma(t))/u_\gamma(t)$ are connected by an explicit
statement rather than by convention.  We claim no novelty for it; its role is
to make precise the sense in which the second family in
\eqref{eq:two-families}, and only that family, deserves the name.

\paragraph{Normalization of the flux.}
A caveat must come first, because it constrains how the zero-flux property
may be read.  Equation \eqref{eq:ivp} is unchanged when $H$ is replaced by
$H+C$ for a constant $C$, but the vanishing of the flux across a single curve
is not: for $u\equiv1$ and $H(u)=u^2/2$ the condition gives slope $1/2$,
whereas $\widetilde H(u)=u^2/2+1$ describes the same equation and gives slope
$3/2$.  What \emph{is} invariant is the difference of the two lateral fluxes
across a fixed pair of curves.  Indeed, replacing $H$ by $H+C$ adds
$C\,(\tau_2-\tau_1)$ to the functional \eqref{eq:fluxfun} for every curve and
every subinterval, so the difference of the two lateral contributions --- and
hence the mass balance of Remark~\ref{rem:tube} --- is unaffected, while the
vanishing of each contribution separately is not.  The zero-flux property of
an individual curve is therefore determined not by the conservation law
alone, but by the conservation law together with a choice of flux
representative.  We fix throughout
\begin{equation}\label{eq:norm}
H(0)=0,
\end{equation}
which is the normalization carried by the models treated in the three sources
compared here: in \cite{DPY2000}, $\Lambda_{\mathrm w}(S)=0$ for
$0<S\le S_{\mathrm{w,res}}$ (eq.~(5.8)), hence $\Lambda_{\mathrm w}(0)=0$; in
the scalar papers, the Burgers, Buckley--Leverett, and
Lighthill--Whitham--Richards traffic
\cite{LighthillWhitham1955,Richards1956} fluxes all vanish at $u=0$.

Two consequences of \eqref{eq:norm} should be kept apart.  It is needed for
the statement that the ratio equation \emph{is} the zero-flux condition ---
that is, for Proposition~\ref{prop:flux} and for the reading of the word
``no-flow.''  It is \emph{not} needed for the comparison of
Section~\ref{sec:three}: Table~\ref{tab:correspondence} and
Figure~\ref{fig:tubes} compare three constructions written with one and the
same flux function, and a change of representative would move all three
alike.

\begin{proposition}[zero normal mass flux]\label{prop:flux}
Let $H$ be locally Lipschitz with $H(0)=0$, let $u$ be a locally bounded
weak solution of \eqref{eq:ivp}, and let $x=\gamma(t)$, $t\in[t_1,t_2]$, be a
Lipschitz curve along which $u$ admits a one-sided trace, denoted by
$u_\gamma\in L^\infty(t_1,t_2)$, from a fixed side.\footnote{\label{fn:traces}The trace
hypothesis is satisfied by piecewise smooth solutions and, more generally, by
weak solutions in $L^\infty_{\mathrm{loc}}\cap BV_{\mathrm{loc}}$, for which
one-sided traces on a Lipschitz graph exist by \cite[Thm.~3.87]{AFP2000}.  In
that case the choice of side is immaterial in \eqref{eq:fluxfun}.  Indeed,
$\operatorname{div}_{t,x}V=0$ as a measure for $V:=(u,H(u))^{\mathrm T}$, and
its absolutely continuous, jump, and Cantor parts are mutually singular, so
each vanishes separately; since the Cantor part does not charge
$\mathcal H^{1}$-$\sigma$-finite sets, the jump part carried by the graph
must vanish.  Hence, at $\mathcal H^{1}$-a.e.\ point of the graph, either the
two traces coincide or the graph is tangent to the jump set of $V$ and
$H(u^{\mathrm l}_\gamma)-\dot\gamma\,u^{\mathrm l}_\gamma
=H(u^{\mathrm r}_\gamma)-\dot\gamma\,u^{\mathrm r}_\gamma$.  Under the weaker
assumption that $V$ is only a divergence-measure field, the relevant
normal-trace theory is that of Chen and Frid \cite{ChenFrid1999}.}
Define the mass flux across $\gamma$ over
$[\tau_1,\tau_2]\subseteq[t_1,t_2]$ by
\begin{equation}\label{eq:fluxfun}
\mathcal F_\gamma(\tau_1,\tau_2)
=\int_{\tau_1}^{\tau_2}
 \bigl[H(u_\gamma(t))-\dot\gamma(t)\,u_\gamma(t)\bigr]\dd t .
\end{equation}
Then $\mathcal F_\gamma(\tau_1,\tau_2)=0$ for every subinterval
$[\tau_1,\tau_2]$ if and only if
\begin{equation}\label{eq:zeroflux}
H\bigl(u_\gamma(t)\bigr)=\dot\gamma(t)\,u_\gamma(t)
\qquad\text{for a.e.\ }t\in[t_1,t_2].
\end{equation}
On $\{t:u_\gamma(t)\ne0\}$, \eqref{eq:zeroflux} is the ratio equation
\begin{equation}\label{eq:ratio}
\dot\gamma(t)=\frac{H(u_\gamma(t))}{u_\gamma(t)} .
\end{equation}
On $\{t:u_\gamma(t)=0\}$, \eqref{eq:zeroflux} holds for every value of
$\dot\gamma$; the zero-flux requirement places no condition on the curve
there.
\end{proposition}

\begin{proof}
Orient the graph of $\gamma$ by the unnormalized normal $(-\dot\gamma,1)$ in
the coordinate order $(t,x)$; since the unit normal is
$(-\dot\gamma,1)/\sqrt{1+\dot\gamma^{2}}$ and the length element is
$\sqrt{1+\dot\gamma^{2}}\dd t$, the integrand $g_\gamma$ of
\eqref{eq:fluxfun} is the density of the flux of $V$ with respect to $\dd t$.
It lies in $L^\infty(t_1,t_2)$ because $u_\gamma$ is bounded, $H$ is locally
Lipschitz, and $\gamma$ is Lipschitz.  If \eqref{eq:zeroflux} holds, then
$g_\gamma=0$ a.e.\ and every integral vanishes.  Conversely, suppose every
integral vanishes and set $G(s):=\int_{t_1}^{s}g_\gamma\dd t$, which is
absolutely continuous.  Then $G(\tau_2)-G(\tau_1)=0$ for all
$t_1\le\tau_1\le\tau_2\le t_2$, so $G$ is constant and the fundamental
theorem of calculus for absolutely continuous functions gives
$g_\gamma=G'=0$ a.e., which is \eqref{eq:zeroflux}.  Dividing by $u_\gamma$
where it is nonzero gives \eqref{eq:ratio}.  Where $u_\gamma=0$,
\eqref{eq:zeroflux} reads $H(0)=0$, which holds by \eqref{eq:norm}, for every
value of $\dot\gamma$.
\end{proof}

\paragraph{Role of the hypotheses.}
They are worth separating.  The equivalence just proved is a property of the
single function $g_\gamma:=H(u_\gamma)-\dot\gamma\,u_\gamma$ and uses only its
integrability.  That $u$ solves \eqref{eq:ivp} is what makes
\eqref{eq:fluxfun} the mass flux of the space--time field
$V=(u,H(u))^{\mathrm T}$ through the graph rather than merely a functional of
$\gamma$, and it is also what makes the value independent of the side from
which the trace is taken (footnote~\ref{fn:traces}).  The normalization
\eqref{eq:norm} is used only in the last sentence of the statement.

\begin{remark}[tube balance]\label{rem:tube}
Let $\gamma_{\mathrm l}<\gamma_{\mathrm r}$ be two such curves, both
satisfying \eqref{eq:zeroflux} with traces taken from inside the region
$E:=\{(t,x):\tau_1<t<\tau_2,\ \gamma_{\mathrm l}(t)<x<\gamma_{\mathrm r}(t)\}$.
Since $\operatorname{div}_{t,x}V=0$, the four oriented boundary fluxes of $V$
on the bounded Lipschitz region $E$ sum to zero.  The outward normal is
$(-\dot\gamma_{\mathrm r},1)$ on the right boundary and
$(\dot\gamma_{\mathrm l},-1)$ on the left, so the two lateral contributions
are $\mathcal F_{\gamma_{\mathrm r}}(\tau_1,\tau_2)$ and
$-\mathcal F_{\gamma_{\mathrm l}}(\tau_1,\tau_2)$, whence
\[
\int_{\gamma_{\mathrm l}(\tau_2)}^{\gamma_{\mathrm r}(\tau_2)}u(x,\tau_2)\dd x
-\int_{\gamma_{\mathrm l}(\tau_1)}^{\gamma_{\mathrm r}(\tau_1)}u(x,\tau_1)\dd x
=\mathcal F_{\gamma_{\mathrm l}}(\tau_1,\tau_2)
-\mathcal F_{\gamma_{\mathrm r}}(\tau_1,\tau_2)=0
\]
for $t_1\le\tau_1\le\tau_2\le t_2$.  Here the horizontal normal traces are
identified with the time slices $u(\cdot,\tau_i)$, as is legitimate for the
entropy solutions in the sense of Kruzhkov \cite{Kruzkov1970} considered in
the three sources, which are continuous in time with values in
$L^1_{\mathrm{loc}}$.  The resulting identity is
eqs.~(5.6)--(5.7) of \cite{DPY2000} in the source-free case,
eqs.~(2.14)--(2.15) of \cite{MPS2007}, and eq.~(4) of \cite{AESLP2023}.
With a source $s\in L^1_{\mathrm{loc}}$ on the right of \eqref{eq:ivp}, the
same computation adds $\iint_E s$ to the right-hand side, recovering
eqs.~(5.6)--(5.7) of \cite{DPY2000} in full; since an $L^1$ source puts no
mass on a Lipschitz graph, the lateral characterization is unchanged.
Sources concentrated on curves or points are not covered.
\end{remark}

\begin{remark}[what Proposition~\ref{prop:flux} does not do]\label{rem:limits}
It starts from a curve that is already given, and characterizes when that
curve has zero normal flux.  It does not assert that such a curve exists
through a given point, that it is unique, that distinct such curves do not
cross, that they depend stably on $u$, or that the curves computed by any of
the discrete algorithms converge to them.  At a discontinuity of $u$, two
distinct issues arise and should not be conflated.  For a curve that is
\emph{given}, the functional \eqref{eq:fluxfun} is unambiguous, by the
compatibility recorded in footnote~\ref{fn:traces}.  But as a prescription
for \emph{producing} a curve, \eqref{eq:ratio} has no canonical pointwise
right-hand side where $u$ jumps: the two one-sided ratios
$H(u^{-})/u^{-}$ and $H(u^{+})/u^{+}$ generally differ from each other, and
generally from the Rankine--Hugoniot speed, so \eqref{eq:ratio} is not a
well-posed ordinary differential equation there.  Each algorithm resolves the
ambiguity by its own reconstruction (cell average, interface value, limited
reconstruction), and we prove nothing about the curves so produced.
Accordingly, we describe the relation between the three constructions as a
correspondence of defining equations, not as an equivalence of methods or of
computed curves.  Finally, a no-flow curve is not in general a shock
trajectory: along a jump satisfying Rankine--Hugoniot the two one-sided normal
fluxes agree, whereas \eqref{eq:zeroflux} additionally requires that this
common flux vanish.  When both traces are nonzero, this is equivalent to both
one-sided ratios being equal to $\dot\gamma$.
\end{remark}

\begin{remark}[vacuum]\label{rem:vacuum}
By the last statement of Proposition~\ref{prop:flux}, if $u\equiv0$ on an open
space--time set then every Lipschitz curve inside it has zero flux, whatever
its slope.  Selecting a continuation through a zero state is therefore a
modeling or numerical decision, not a consequence of the zero-flux
requirement.  Under \eqref{eq:norm} and local Lipschitz continuity of $H$ one
has $|H(u)/u|=|H(u)-H(0)|/|u|\le\operatorname{Lip}(H)$ near $u=0$, so the
issue is the assignment of a value to an indeterminate quotient rather than
an unbounded slope; it is nonetheless a real one for a scheme that must
evaluate the ratio at every tracked point.  The article under comment defines
the positive and negative parts $f^+=\max(f,0)$ and $f^-=\max(-f,0)$ only
after setting $f=H(u)/u$ \cite[p.~2407, eq.~(11)]{AESLP2023}.  This upwind splitting does not itself assign
$f(0)$; the convergence analysis expressly assumes $u\ne0$, indeed
$u>a>0$ \cite[p.~2409, immediately after eq.~(19)]{AESLP2023}; the no-flow
ODE is itself introduced there ``assuming $u\neq0$ (for the sake of
presentation)'' \cite[p.~2405]{AESLP2023}.  For Burgers' flux $H(u)=u^2/2$, the quotient instead admits the
removable extension $f(u)=u/2$ through zero.  A problem-dependent flux
decomposition such as the strategy called ``flux separation'' in
\cite[p.~2314, Section~3, eq.~(19)]{AbreuPerez2019} is a distinct operation.
\end{remark}

\section{Chronology of the terminology}\label{sec:chronology}

The object described in Section~\ref{sec:three} appears in the pertinent
literature in the following sequence.  In 2000, \cite{DPY2000} established the
space--time integral curves and the tube they bound, the zero lateral flux, the
exact local balance, the contrast with characteristic curves, and the
non-degeneracy at vanishing states.  In 2007, \cite{MPS2007,MPS2007b}
specialized the construction to nonlinear scalar conservation laws in
forward-tracked form, naming the curves \emph{curvas integrais} and the region
\emph{tubo integral} and citing \cite{DPY2000}; \cite{QPAS2007} made forward
integral-tube tracking explicit in the linear advective setting.  From 2017
through the HYP2018 proceedings published in 2020, the first papers of the
later Lagrangian--Eulerian series
\cite{ALPS2017,APS-CNMAC2017,APS-CNMAC2018,APS2018,ALPS-HYP2020} built
finite-volume schemes on space--time control volumes bounded by parameterized
integral curves, imposed zero lateral flux, and cited \cite{DPY2000}.  The
2019 paper \cite{AbreuPerez2019} described a ``space--time tracking forward''
scheme, labeled the geometry ``Integral Tube'' in its schematic figure, and
credited the concept to \cite{DPY2000}.
The caption of its Figure~5 already calls the tracked boundaries ``the no flow
curves of the integral tubes.''  The later terminological change is therefore
not the first appearance of the phrase, but its promotion from a description
of the boundaries within the integral-tube construction to the dominant name
of the construction.

Table~\ref{tab:attrib} sets out, for every paper in the defined corpus, the
sentence or passage in which the construction is attributed.  What changes
across the sequence is not the defining continuous geometry ---
Table~\ref{tab:correspondence} shows the same curve equation, tube, and
zero-flux property --- but the account given of where it came from.

\begingroup
\footnotesize
\renewcommand{\arraystretch}{1.15}
\setlength{\LTpre}{0.7\baselineskip}
\setlength{\LTpost}{0.7\baselineskip}
\begin{longtable}{%
>{\raggedright\arraybackslash}p{0.20\textwidth}
>{\raggedright\arraybackslash}p{0.14\textwidth}
>{\raggedright\arraybackslash}p{0.55\textwidth}}
\caption{How the construction is attributed in the defined full-text corpus.
Text in quotation marks reproduces the source wording; bracketed numbers
inside quotations are the original reference numbers of the paper quoted,
retained exactly and identified in the corresponding entry; bracketed
ellipses mark omissions.  Panel~A lists
papers in which \cite{DPY2000} is identified as a source of the construction,
alone or among named antecedents without a conflicting novelty claim.
Panel~B lists mixed, incomplete, self-directed, or absent attribution,
including papers that cite DPY while presenting the same construction, or a
defining component of it, as new or different.}
\label{tab:attrib}\\
\toprule
\textbf{Paper} & \textbf{Name used} & \textbf{Attribution, as stated there}\\
\midrule
\endfirsthead
\multicolumn{3}{@{}l}{\footnotesize\tablename~\thetable\ (continued)}\\[2pt]
\toprule
\textbf{Paper} & \textbf{Name used} & \textbf{Attribution, as stated there}\\
\midrule
\endhead
\midrule
\multicolumn{3}{r@{}}{\footnotesize Continued on the next page}\\
\endfoot
\bottomrule
\endlastfoot
\multicolumn{3}{@{}>{\raggedright\arraybackslash}p{0.90\textwidth}@{}}{%
\textbf{A. \cite{DPY2000} is identified as a source of the construction}}\\
\midrule
\cite{ALPS2017} (2017) & integral curve / integral tube &
Introduction, p.~3: ``In the work [14], the authors identified the region in
the space--time domain where the mass conservation takes place, but linked to
a scalar convection-dominated nonlinear parabolic problem''; reference [14]
there is \cite{DPY2000}.  The paper then calls $\sigma_j^n$ an integral curve
and sets $\dot{\sigma}_j^n=H(u)/u$\\
\cite{APS-CNMAC2017} (2017) & parameterized integral curve &
Introduction, p.~010329-1: ``In the work [4], it was identified the region in
the space-time domain where the mass conservation takes place''; reference
[4] there is \cite{DPY2000}.  The lateral boundaries are then called integral
curves, required to have zero flux, and set to satisfy
$\dot\sigma=H(u)/u$\\
\cite{APS-CNMAC2018} (2018) & parameterized integral curve / space--time
volume &
Abstract, p.~010296-1: ``As in [3,5] the mass conservation takes place in the
space-time volume $D_j^n$''; references [3] and [5] there are
\cite{APS-CNMAC2017} and \cite{DPY2000}.  The paper again defines its lateral
boundaries as parameterized integral curves with zero flux\\
\cite{APS2018} (2018) & integral tube &
Section~2, p.~3: ``Following [1, 2, 14, 28, 20]''; reference [14] there is
\cite{DPY2000}.  After imposing impermeability, ``the region $D_j^n$ will be
called `Integral tube' ''\\
\cite{AbreuPerez2019} (2019) & integral tube; ``no flow curves'' in Fig.~5 &
Introduction, p.~2310: ``the authors also used the concept of integral tubes
(or space--time tubes) introduced in [17]''; reference [17] there is
\cite{DPY2000}\\
\cite{ALPS-HYP2020} (HYP2018; publ.~2020) & parameterized integral curve /
integral tube &
Introduction, p.~223: ``In the work [11], it was identified the region in the
space-time domain where the mass conservation takes place''; reference [11]
there is \cite{DPY2000}.  The paper imposes zero flux, sets
$\dot\sigma=H(u)/u$, and captions Figure~1 ``The Integral tube''\\
\cite{APL2025bal} (2025) & No-Flow curves &
Introduction, p.~2: ``the idea of integral curves as introduced in [24]
for the treatment of nonlinear transport in porous media''; reference [24]
there is \cite{DPY2000}\\
\midrule
\multicolumn{3}{@{}>{\raggedright\arraybackslash}p{0.90\textwidth}@{}}{%
\textbf{B. Presented as new or different, or not attributed to
\cite{DPY2000}}}\\
\midrule
\cite{AMPR2021} (2021) & no-flow region &
Section~2.1, p.~6: ``in the terminology discussed in [24,25], the region
$D^n_j$ is called integral tube, which is substantially different from our
recent developments [1,4--8,48] supported with very important differences in
concept of the no-flow region per time step and theory foundations'';
references [24] and [25] there are \cite{DouglasHuang2001} and
\cite{DPY2000}, respectively\\
\cite{AFLP-JCAM2022} (2022) & no-flow curves &
Introduction, p.~3: ``We present a new 1D SDLE scheme based on the novel
concept of no-flow curves recently introduced in the literature [5,6,28]'';
those three references are to papers by overlapping author teams.  On the
same page, the lateral boundaries are ``two special curves [27] (here under
the name no-flow curves [6])''; references [27] and [6] there are
\cite{DPY2000} and \cite{AbreuPerez2019}, respectively\\
\cite{AFLP-JSC2022} (2022) & no flow curves &
Abstract, p.~1: ``the space--time no flow surface region, previously presented
and analyzed by the authors for fully-discrete schemes''; Introduction, p.~2:
``the space--time no flow curves, per time step, previously presented and
studied by the authors for fully-discrete schemes [1--8]''; references [1--8]
there are all self-citations\\
\cite{AESLP2023} (2023) & no-flow curves &
abstract: ``the improved concept of no-flow curves, as introduced by the
authors''; pp.~2401 and 2405: ``based on the new and substantial improvement
interpretation of the integral tube''; p.~2402: ``novel fully discrete schemes
based on the new concept no-flow curves [\dots] which is substantially
different in theory foundations from the previous and relevant LCELM method.''
Page~2401 nevertheless credits the LCELM and cites eqs.~(5.4a)--(5.4b) of
reference [27] there, namely \cite{DPY2000}\\
\cite{AdlCLP2022} (2024) & no-flow curves &
Section~2.1, p.~1440: ``the new and substantial improvement interpretation
of the integral tube, which is now subject to condition [\dots] (called
no-flow curves [3])''; reference [3] there is \cite{AbreuPerez2019}\\
\cite{ACDJL2024lwr} (2024) & no-flow curves &
Section~2,  p.~5: ``the concept of the no-flow property introduced in [2]'';
reference [2] there is \cite{AbreuPerez2019}.  On PDF p.~4,
\cite{DPY2000}, reference [15] there, appears only in a list ``for nonlinear
scalar transport problems''\\
\cite{AAP2024tri} (2024) & no-flow surfaces &
Introduction, p.~2: ``Lagrangian methods can be based on the concept of
space--time control volume. This concept was first introduced in [34--36] to
solve parabolic convection--diffusion transport models (scalar case)'';
references [34], [35], and [36] there are \cite{DPY2000}, \cite{DPY2000b}, and
\cite{DouglasHuang2001}.  The same page says that the no-flow-surface concept
``has been extensively explored and analyzed by the authors''; Section~2,
p.~4, begins ``Based on the novel concept of Lagrangian--Eulerian no-flow
curves (per time step)''\\
\cite{AALP2025tri} (2025) & no-flow curves &
Appendix~A.1, p.~808: ``As in [6--9, 18, 28, 29], we assume that
$[u,H(u)]\cdot\vec n=0$''; reference [18] there is \cite{DPY2000}, bundled
with six self-citations, and DPY is nowhere named as the source of the curves\\
\cite{AdlCJL2025} (2025) & no-flow curves &
no reference to \cite{DPY2000} appears anywhere in the paper.  The abstract,
p.~1, says that the authors ``expanded upon the (local) semi-discrete
Lagrangian--Eulerian method initially introduced in Abreu et al. (2022)'';
on p.~4, for ``a detailed description of no-flow curves,'' the reader is
referred to reference [21] there, namely \cite{AdlCLP2022}\\
\cite{ALL2026dd} (2026) & no-flow curves / no-flow surfaces &
Remark~1, p.~8, says that the no-flow curve ``naturally extends the integral
curve originally introduced by Douglas, Pereira, and Yeh [38]'' and also, ``In
the present paper, we adopt this established construction and apply it to
equations of the form (1).''  Definition~2.1 and eqs.~(9), (10), and (14)
reproduce the same space--time zero-flux field and $\dot\sigma=H/u$ curve;
p.~10 states that diffusion and dispersion do not modify the defining vector
field, while p.~8 calls its space--time vector representation ``novel''\\
\cite{AFFG2026} (2026) & multi\-dimensional no-flow curves &
abstract, p.~1: ``a recently introduced semi-discrete Lagrangian--Eulerian
formulation for the hyperbolic transport equations''; p.~2: the ``recently
developed genuinely multidimensional'' scheme ``is based on `no-flow' curves
[9],'' where reference [9] there is \cite{AbreuPerez2019}.  Equation~(5.46),
p.~11, uses the componentwise ratios $H_L(S)/S_L$ and $G_L(S)/S_L$; no
reference to \cite{DPY2000} appears anywhere in the paper\\
\end{longtable}
\endgroup

Three remarks on the table.  First, it is not merely a list of omissions:
\cite{DPY2000} is cited in every paper listed except
\cite{AdlCJL2025,AFFG2026}.  Panel~B records several distinct defects ---
absent attribution, attribution confined to a list or a narrower setting, or a
direct acknowledgment accompanied by a claim that the identical construction
is new, different, or an extension.
Second, Panel~A shows that \cite{DPY2000} was known and expressly
identified as a source of the construction in papers by overlapping author
teams.  Third, the transitional terminology in \cite{AbreuPerez2019} does not
reconcile the two panels: that
paper refers to ``the no flow curves of the integral tubes'' while crediting
the concept of integral tubes to \cite{DPY2000}.  The same construction is
described as substantially different from the one in \cite{DPY2000} in
\cite{AMPR2021}, and as introduced by the authors in \cite{AESLP2023},
although the defining continuous curve equation and zero-flux property did
not change in between.

\section{What is new in the article under comment}\label{sec:new}

To keep that conclusion within its proper scope, we now state explicitly what
remains new in the article under comment.  Within the documented lineage, the
contributions of \cite{AESLP2023} that we
do not find in \cite{DPY2000}, \cite{MPS2007,MPS2007b}, or \cite{QPAS2007}
are: the particular fully discrete \emph{non-staggered} scheme, with dynamic
forward evolution followed by projection onto the fixed Eulerian grid --- the
evolution-and-projection template itself being classical, as
Section~\ref{sec:which} records, so that what is new here is the
non-staggered realization of it; a weak
CFL condition formulated through $|H(u)/u|$ alone, where the 2007 restriction
(eq.~(2.21) of \cite{MPS2007}) involved both $f'$ and the ratio $f(U)/U$; and
a convergence theory built on the weak asymptotic method of Danilov,
Omel'yanov, and Shelkovich \cite{DOS2003,DanilovShelkovich2005}, yielding
bounded-variation, maximum-principle, and Kruzhkov-entropy conclusions.  We do
not challenge the novelty or value of those contributions here.  The companion
papers extend the framework to forcing terms, multidimensional systems,
nonlocal laws, triangular grids, discontinuous-flux models, and multiscale
porous-media simulations.

Two qualifications belong alongside this list, and they are qualifications of
scope rather than of merit.  First, forward tracking, fully discrete
evolution of cell averages, and limited piecewise-linear reconstruction
already occur in the staggered FLCELM and FLCELM-R schemes of
\cite{MPS2007,MPS2007b}, and forward integral-tube tracking for linear
advection in \cite{QPAS2007}; what is new is the non-staggered realization
and its analysis, not forward tracking as such.  Second, the weak asymptotic
method is not itself a contribution of \cite{AESLP2023}: it was developed by
Danilov, Omel'yanov, and Shelkovich \cite{DOS2003,DanilovShelkovich2005}, and
\cite{AESLP2023} cites that work.  What belongs to \cite{AESLP2023} is the
application of the method to the scheme it constructs, and that application
is, of everything in the article, the part least anticipated by the earlier
literature and the part on which its claim to novelty rests most securely.  It
concerns the discretization, and nothing in this Comment touches it.

\section{Conclusion}\label{sec:conclusion}

\paragraph{Two statements from the same article.}
The article under comment states, on page~2401, that \cite{DPY2000} was ``the
first work in the literature to introduce a space--time local conservation,''
with the region called the integral tube and bounded by integral curves, and
refers there to equations (5.4a)--(5.4b) of \cite{DPY2000} --- the
integral-curve problem \eqref{eq:dpy-curve} itself.  Its
abstract states that its scheme rests on ``the improved concept of no-flow
curves, as introduced by the authors,'' and page~2402 describes that concept
as ``substantially different in theory foundations'' from the LCELM.
Sections~\ref{sec:three} and \ref{sec:which} establish that the two
statements are about the same continuous object: under the dictionary
\eqref{eq:dict}, the curve of eq.~(6) of \cite{AESLP2023} is the curve of
eq.~(5.4a) of \cite{DPY2000}, the no-flow region is the tube $\mathcal D$, the
zero-flux property is the orthogonality of the normal to
$(\Phi S,\Lambda_{\mathrm w}U)^{\mathrm T}$, and the local balance (4) is
(5.6)--(5.7).  Within the documented Eulerian--Lagrangian lineage, the first
statement is consistent with the record.  The second, understood as a claim
about the continuous curve rather than about the discrete scheme built upon
it, is not supported by that record and is in tension with the first.  This
identification is also confirmed by the equations in \cite{ALL2026dd}, which
derive the same ratio ODE and state that diffusion and dispersion leave the
curve-defining vector field unchanged.  More broadly, the chronology and
Table~\ref{tab:attrib} document a series of papers in which the same
construction is variously credited to DPY, associated principally with later
work, or presented as new, different, or an extension.

\paragraph{What the correspondence does and does not settle.}
It is a correspondence of defining equations of continuous objects.  It
asserts nothing about existence, uniqueness, or selection of curves, nothing
about the discrete curves the algorithms compute (Remark~\ref{rem:limits}),
and nothing about moving-mesh methods in general
(Section~\ref{sec:which}).  In particular it leaves intact everything set out
in Section~\ref{sec:new}: the non-staggered fully discrete design, the
weak CFL condition, and the convergence analysis carried out for that scheme
by the weak asymptotic method of \cite{DOS2003,DanilovShelkovich2005}, all of
which remain contributions of
\cite{AESLP2023} whatever the continuous geometry is called.  Our conclusion
concerns a name and the object it names, not the value of the work done with
them.

\paragraph{What a reader gains.}
Orientation, in the first place: a reader meeting the no-flow curves can
locate the same continuous object in Section~5 of \cite{DPY2000} and in the
2007 scalar papers, and read those works as part of one development rather
than as unrelated antecedents cited for a single property.  Transfer, in the
second: because the continuous geometry is shared, the zero-flux
characterization of Proposition~\ref{prop:flux}, the tube balance of
Remark~\ref{rem:tube}, and the vacuum discussion of Remark~\ref{rem:vacuum}
may be read once and applied to either formulation.  And a sharper view of the
novelty itself, in the third: separating the continuous geometry from the
discretization makes clearer, not less clear, where the contributions of
\cite{AESLP2023} lie.  We ask nothing further than that the continuous
construction be located where the published record places it.

\subsection*{Statements and declarations}

{\small
\noindent\emph{Funding.} F.~Pereira is partially supported by National Science
Foundation (USA) grant 2401945.  Any opinions, findings, and conclusions or
recommendations expressed in this material are those of the authors and do not
necessarily reflect the views of the National Science Foundation.

\smallskip
\noindent\emph{Competing interests.} F.~Pereira is a co-author of
\cite{DPY2000,MPS2007,MPS2007b,QPAS2007} and L.-M.~Yeh is a co-author of
\cite{DPY2000}; the first author of \cite{AESLP2023} was formerly
F.~Pereira's doctoral student.  The authors declare no other competing
interests.  This Comment documents the content of published sources; the
comparisons on which it rests are stated as equation numbers and quotations
so that any reader may check them independently of the authors.

\smallskip
\noindent\emph{Author contributions (CRediT).} Conceptualization,
investigation, writing --- original draft, and writing --- review and editing:
F.~Furtado, F.~Pereira, L.-M.~Yeh (equal contributions).

\smallskip
\noindent\emph{Data availability.} No datasets were generated or analyzed; all
material quoted or cited is available in the published sources listed in the
references.

\smallskip
\noindent\emph{Declaration of generative AI and AI-assisted technologies in
the writing process.} During the preparation of this work, the authors used
OpenAI's ChatGPT and Codex to assist with language editing, stylistic
polishing, document organization, consistency and cross-reference checks,
LaTeX preparation, and drafting of the schematic figure.  After using these
tools, the authors reviewed and edited the resulting material as needed and
take full responsibility for the content of the manuscript.
\par}

\begingroup
\footnotesize
\raggedright

\endgroup


\begin{thebibliography}{99}
\setlength{\itemsep}{0.12em}

\bibitem{AALP2025tri}
E.~Abreu, J.~Agudelo, W.~Lambert, and J.~P\'erez,
\emph{A Lagrangian--Eulerian method on regular triangular grids for
hyperbolic problems: error estimates for the scalar case and a positive
principle for multidimensional systems},
J.\ Dyn.\ Differ.\ Equ.\ \textbf{37} (2025), no.~1, 749--814.
\mbox{doi:10.1007/s10884-023-10283-1}.

\bibitem{AFFG2026}
E.~Abreu, P.~Ferraz, J.~R. Fran\c{c}ois, and J.~Galvis,
\emph{Integrating semi-discrete Lagrangian--Eulerian schemes with generalized
multiscale finite elements for enhanced two- and three-phase flow simulations},
J.\ Comput.\ Appl.\ Math.\ \textbf{484} (2026), art.~117401.
\mbox{doi:10.1016/j.cam.2026.117401}.

\bibitem{ALL2026dd}
E.~Abreu, W.~Lambert, and E.~Lima,
\emph{A semi-discrete Lagrangian--Eulerian numerical scheme for
diffusive--dispersive conservation laws with discontinuous coefficient},
Numer.\ Methods Partial Differential Equations \textbf{42} (2026),
art.~e70131. \mbox{doi:10.1002/num.70131}.

\bibitem{AAP2024tri}
E.~Abreu, J.~Agudelo, and J.~P\'erez,
\emph{A triangle-based positive semi-discrete Lagrangian--Eulerian scheme via
the weak asymptotic method for scalar equations},
J.\ Comput.\ Appl.\ Math.\ \textbf{437} (2024), art.~115465.
\mbox{doi:10.1016/j.cam.2023.115465}.

\bibitem{ACDJL2024lwr}
E.~Abreu, M.~T. Chiri, R.~De la Cruz, J.~Juajibioy, and W.~Lambert,
\emph{A semidiscrete Lagrangian--Eulerian scheme for the LWR traffic model
with discontinuous flux}, arXiv:2412.06692 (2024).

\bibitem{AdlCLP2022}
E.~Abreu, R.~De la Cruz, J.~C. Juajibioy, and W.~Lambert,
\emph{Lagrangian--Eulerian approach for nonlocal conservation laws},
J.\ Dyn.\ Differ.\ Equ.\ \textbf{36} (2024), 1435--1481.
\mbox{doi:10.1007/s10884-022-10193-8}.

\bibitem{AdlCJL2025}
E.~Abreu, R.~De la Cruz, J.~Juajibioy, and W.~Lambert,
\emph{Semi-discrete Lagrangian--Eulerian approach based on the weak asymptotic
method for nonlocal conservation laws in several dimensions},
J.\ Comput.\ Appl.\ Math.\ \textbf{458} (2025), art.~116325.
\mbox{doi:10.1016/j.cam.2024.116325}.

\bibitem{AESLP2023}
E.~Abreu, A.~Esp\'{\i}rito Santo, W.~Lambert, and J.~P\'erez,
\emph{Convergence, bounded variation properties and Kruzhkov solution of a
fully discrete Lagrangian--Eulerian scheme via weak asymptotic analysis for
1D hyperbolic problems},
Numer.\ Methods Partial Differential Equations \textbf{39} (2023), no.~3,
2400--2443. \mbox{doi:10.1002/num.22972}.  Preprint version:
arXiv:\hspace{0pt}2106.08363v3 [math.NA], 2 February 2022, under the title
\emph{Convergence of a Lagrangian--Eulerian scheme by a weak asymptotic
analysis for one-dimensional hyperbolic problems}.

\bibitem{AFLP-JCAM2022}
E.~Abreu, J.~Fran\c{c}ois, W.~Lambert, and J.~P\'erez,
\emph{A semi-discrete Lagrangian--Eulerian scheme for hyperbolic-transport
models},
J.\ Comput.\ Appl.\ Math.\ \textbf{406} (2022), art.~114011.
\mbox{doi:10.1016/j.cam.2021.114011}.

\bibitem{AFLP-JSC2022}
E.~Abreu, J.~Fran\c{c}ois, W.~Lambert, and J.~P\'erez,
\emph{A class of positive semi-discrete Lagrangian--Eulerian schemes for
multidimensional systems of hyperbolic conservation laws},
J.\ Sci.\ Comput.\ \textbf{90} (2022), no.~1, art.~40.
\mbox{doi:10.1007/s10915-021-01712-8}.

\bibitem{ALPS2017}
E.~Abreu, W.~Lambert, J.~P\'erez, and A.~Santo,
\emph{A new finite volume approach for transport models and related
applications with balancing source terms},
Math.\ Comput.\ Simulat.\ \textbf{137} (2017), 2--28.
\mbox{doi:10.1016/j.matcom.2016.12.012}.

\bibitem{ALPS-HYP2020}
E.~Abreu, W.~Lambert, J.~P\'erez, and A.~Santo,
\emph{A weak asymptotic solution analysis for a Lagrangian--Eulerian scheme
for scalar hyperbolic conservation laws},
in: A.~Bressan, M.~Lewicka, D.~Wang, and Y.~Zheng (eds.),
Hyperbolic Problems: Theory, Numerics, Applications, AIMS Ser.\ Appl.\ Math.,
vol.~10, American Institute of Mathematical Sciences, Springfield, MO, 2020,
223--230.

\bibitem{AMPR2021}
E.~Abreu, V.~Matos, J.~P\'erez, and P.~Rodr\'{\i}guez-Berm\'udez,
\emph{A class of Lagrangian--Eulerian shock-capturing schemes for first-order
hyperbolic problems with forcing terms},
J.\ Sci.\ Comput.\ \textbf{86} (2021), art.~14.
\mbox{doi:10.1007/s10915-020-01392-w}.

\bibitem{APL2025bal}
E.~Abreu, E.~Pandini, and W.~Lambert,
\emph{An enhanced Lagrangian--Eulerian method for a class of balance laws:
numerical analysis via a weak asymptotic method with applications},
Numer.\ Methods Partial Differential Equations \textbf{41} (2025), no.~1,
art.~e23163. \mbox{doi:10.1002/num.23163}.

\bibitem{AbreuPerez2019}
E.~Abreu and J.~P\'erez,
\emph{A fast, robust, and simple Lagrangian--Eulerian solver for balance laws
and applications},
Comput.\ Math.\ Appl.\ \textbf{77} (2019), 2310--2336.
\mbox{doi:10.1016/j.camwa.2018.12.019}.

\bibitem{APS2018}
E.~Abreu, J.~P\'erez, and A.~Santo,
\emph{Lagrangian--Eulerian approximation methods for balance laws and
hyperbolic conservation laws},
Rev.\ UIS Ing.\ \textbf{17} (2018), no.~1, 191--200.
\mbox{doi:10.18273/revuin.v17n1-2018018}.

\bibitem{APS-CNMAC2017}
E.~Abreu, J.~P\'erez, and A.~Santo,
\emph{Solving hyperbolic conservation laws by using Lagrangian--Eulerian
approach},
Proceeding Series of the Brazilian Society of Computational and Applied
Mathematics \textbf{5} (2017), no.~1, 010329-1--010329-7.
\mbox{doi:10.5540/03.2017.005.01.0329}.

\bibitem{APS-CNMAC2018}
E.~Abreu, J.~P\'erez, and A.~Santo,
\emph{A conservative Lagrangian--Eulerian finite volume approximation method
for balance law problems},
Proceeding Series of the Brazilian Society of Computational and Applied
Mathematics \textbf{6} (2018), no.~1, 010296-1--010296-7.
\mbox{doi:10.5540/03.2018.006.01.0296}.

\bibitem{ALLEA2023}
ALLEA -- All European Academies,
\emph{The European Code of Conduct for Research Integrity}, revised ed.,
Berlin, 2023. \mbox{doi:10.26356/ECOC}.

\bibitem{AFP2000}
L.~Ambrosio, N.~Fusco, and D.~Pallara,
\emph{Functions of Bounded Variation and Free Discontinuity Problems},
Oxford Math.\ Monogr., Oxford University Press, Oxford, 2000.

\bibitem{QPAS2007}
J.~Aquino, F.~Pereira, H.~P. Amaral Souto, and A.~S. Francisco,
\emph{A forward tracking scheme for solving radionuclide advective problems in
unsaturated porous media},
Int.\ J.\ Nuclear Energy Sci.\ Technol.\ \textbf{3} (2007), no.~2, 196--205.
\mbox{doi:10.1504/IJNEST.2007.014656}.

\bibitem{ArbogastHuang2010}
T.~Arbogast and C.-S. Huang,
\emph{A fully conservative Eulerian--Lagrangian method for a
convection--diffusion problem in a solenoidal field},
J.\ Comput.\ Phys.\ \textbf{229} (2010), 3415--3427.
\mbox{doi:10.1016/j.jcp.2010.01.009}.

\bibitem{ArbogastHuangHung2012}
T.~Arbogast, C.-S. Huang, and C.-H. Hung,
\emph{A fully conservative Eulerian--Lagrangian stream-tube method for
advection--diffusion problems},
SIAM J.\ Sci.\ Comput.\ \textbf{34} (2012), no.~4, B447--B478.
\mbox{doi:10.1137/110840376}.

\bibitem{ArbogastWheeler1995}
T.~Arbogast and M.~F. Wheeler,
\emph{A characteristics-mixed finite element method for advection-dominated
transport problems},
SIAM J.\ Numer.\ Anal.\ \textbf{32} (1995), no.~2, 404--424.
\mbox{doi:10.1137/0732017}.

\bibitem{BuckleyLeverett1942}
S.~E. Buckley and M.~C. Leverett,
\emph{Mechanism of fluid displacement in sands},
Trans.\ AIME \textbf{146} (1942), 107--116.
\mbox{doi:10.2118/942107-G}.

\bibitem{CRHE1990}
M.~A. Celia, T.~F. Russell, I.~Herrera, and R.~E. Ewing,
\emph{An Eulerian--Lagrangian localized adjoint method for the
advection--diffusion equation},
Adv.\ Water Resour.\ \textbf{13} (1990), no.~4, 187--206.
\mbox{doi:10.1016/0309-1708(90)90041-2}.

\bibitem{ChenFrid1999}
G.-Q. Chen and H.~Frid,
\emph{Divergence-measure fields and hyperbolic conservation laws},
Arch.\ Ration.\ Mech.\ Anal.\ \textbf{147} (1999), 89--118.

\bibitem{ColellaWoodward1984}
P.~Colella and P.~R. Woodward,
\emph{The piecewise parabolic method (PPM) for gas-dynamical simulations},
J.\ Comput.\ Phys.\ \textbf{54} (1984), no.~1, 174--201.
\mbox{doi:10.1016/0021-9991(84)90143-8}.

\bibitem{DOS2003}
V.~G. Danilov, G.~A. Omel'yanov, and V.~M. Shelkovich,
\emph{Weak asymptotics method and interaction of nonlinear waves},
in: M.~Karasev (ed.), Asymptotic Methods for Wave and Quantum Problems,
Amer.\ Math.\ Soc.\ Transl.\ Ser.~2, vol.~208, American Mathematical
Society, Providence, RI, 2003, 33--165.
\mbox{doi:10.1090/trans2/208/02}.

\bibitem{DanilovShelkovich2005}
V.~G. Danilov and V.~M. Shelkovich,
\emph{Dynamics of propagation and interaction of $\delta$-shock waves in
conservation law systems},
J.\ Differential Equations \textbf{211} (2005), no.~2, 333--381.
\mbox{doi:10.1016/j.jde.2004.12.011}.

\bibitem{DFP1997}
J.~Douglas, Jr., F.~Furtado, and F.~Pereira,
\emph{On the numerical simulation of waterflooding of heterogeneous petroleum
reservoirs},
Comput.\ Geosci.\ \textbf{1} (1997), no.~2, 155--190.
\mbox{doi:10.1023/A:1011565228179}.

\bibitem{DHP1999}
J.~Douglas, Jr., C.-S. Huang, and F.~Pereira,
\emph{The modified method of characteristics with adjusted advection},
Numer.\ Math.\ \textbf{83} (1999), 353--369.
\mbox{doi:10.1007/s002110050453}.

\bibitem{DPY2000}
J.~Douglas, Jr., F.~Pereira, and L.-M. Yeh,
\emph{A locally conservative Eulerian--Lagrangian numerical method and its
application to nonlinear transport in porous media},
Comput.\ Geosci.\ \textbf{4} (2000), 1--40.
\mbox{doi:10.1023/A:1011551614492}.

\bibitem{DPY2000b}
J.~Douglas, Jr., F.~Pereira, and L.-M. Yeh,
\emph{A locally conservative Eulerian--Lagrangian method for flow in a porous
medium of a mixture of two components having different densities},
in: Z.~Chen, R.~E. Ewing, and Z.-C. Shi (eds.), Numerical Treatment of
Multiphase Flows in Porous Media, Lecture Notes in Physics, vol.~552,
Springer, Berlin, Heidelberg, 2000, 138--155.

\bibitem{DouglasHuang2001}
J.~Douglas, Jr.\ and C.-S. Huang,
\emph{A locally conservative Eulerian--Lagrangian finite difference method
for a parabolic equation},
BIT Numer.\ Math.\ \textbf{41} (2001), no.~3, 480--489.

\bibitem{HirtAmsdenCook1974}
C.~W. Hirt, A.~A. Amsden, and J.~L. Cook,
\emph{An arbitrary Lagrangian--Eulerian computing method for all flow speeds},
J.\ Comput.\ Phys.\ \textbf{14} (1974), no.~3, 227--253.
\mbox{doi:10.1016/0021-9991(74)90051-5}.

\bibitem{DouglasRussell1982}
J.~Douglas, Jr.\ and T.~F. Russell,
\emph{Numerical methods for convection-dominated diffusion problems based on
combining the method of characteristics with finite element or finite
difference procedures},
SIAM J.\ Numer.\ Anal.\ \textbf{19} (1982), no.~5, 871--885.
\mbox{doi:10.1137/0719063}.

\bibitem{Kruzkov1970}
S.~N. Kruzhkov,
\emph{First order quasilinear equations in several independent variables},
Math.\ USSR Sb.\ \textbf{10} (1970), no.~2, 217--243.
\mbox{doi:10.1070/SM1970v010n02ABEH002156}.

\bibitem{LighthillWhitham1955}
M.~J. Lighthill and G.~B. Whitham,
\emph{On kinematic waves. II. A theory of traffic flow on long crowded roads},
Proc.\ Roy.\ Soc.\ London Ser.\ A \textbf{229} (1955), 317--345.
\mbox{doi:10.1098/rspa.1955.0089}.

\bibitem{MPS2007b}
S.~Mancuso, F.~Pereira, and G.~de Souza,
\emph{Adaptividade de malhas na aproxima\c{c}\~ao lagrangeana de leis de
conserva\c{c}\~ao} [Mesh adaptivity in the Lagrangian approximation of
conservation laws],
TEMA Tend.\ Mat.\ Apl.\ Comput.\ \textbf{8} (2007), no.~2, 269--276
(in Portuguese). \mbox{doi:10.5540/tema.2007.08.02.0269}.

\bibitem{MPS2007}
S.~Mancuso, F.~Pereira, and G.~de Souza,
\emph{Um novo m\'etodo euleriano-lagrangeano para aproxima\c{c}\~ao de leis de
conserva\c{c}\~ao} [A new Eulerian--Lagrangian method for the approximation of
conservation laws],
TEMA Tend.\ Mat.\ Apl.\ Comput.\ \textbf{8} (2007), no.~2, 277--286
(in Portuguese). \mbox{doi:10.5540/tema.2007.08.02.0277}.

\bibitem{Richards1956}
P.~I. Richards,
\emph{Shock waves on the highway},
Oper.\ Res.\ \textbf{4} (1956), no.~1, 42--51.
\mbox{doi:10.1287/opre.4.1.42}.

\bibitem{vanLeer1979}
B.~van Leer,
\emph{Towards the ultimate conservative difference scheme. V. A second-order
sequel to Godunov's method},
J.\ Comput.\ Phys.\ \textbf{32} (1979), no.~1, 101--136.
\mbox{doi:10.1016/0021-9991(79)90145-1}.

\bibitem{WileyEthics}
John Wiley \& Sons,
\emph{Best Practice Guidelines on Publishing Ethics: A Publisher's
Perspective}, 2nd ed., 2014; see also Wiley's current publication ethics
guidelines for journals,
\url{https://authors.wiley.com/ethics-guidelines/}.

\end{thebibliography}
\end{document}